\documentclass[12pt]{amsart}
\usepackage[all]{xy}
\usepackage{amssymb}
\usepackage{amsthm}
\usepackage{amsmath}
\usepackage{amscd,enumitem}
\usepackage{verbatim}
\usepackage{eurosym}
\usepackage{graphicx}
\usepackage{dsfont}
\usepackage{stmaryrd}
\usepackage{color}
\usepackage{float}
\usepackage{longtable}
\usepackage{dcolumn}
\usepackage[mathscr]{eucal}
\usepackage[all]{xy}
\usepackage{hyperref}
\usepackage[title]{appendix}
\usepackage[usenames,dvipsnames]{xcolor}
\usepackage{bbm}
\usepackage[textheight=8.85in, textwidth=6.8in]{geometry}
\usepackage{multirow}
\usepackage{caption}

\theoremstyle{plain}
\newtheorem{theorem}{Theorem}[section]
\newtheorem{corollary}[theorem]{Corollary}
\newtheorem{lemma}[theorem]{Lemma}

\theoremstyle{definition}
\newtheorem{definition}[theorem]{Definition}

\theoremstyle{remark}
\newtheorem*{remark}{Remark}

\newcommand{\F}{\mathbb{F}}

\newcommand{\re}{\mathrm{Re}}

\newcommand{\Tr}{{\mathrm{Tr}_k}}

\newcommand{\Leg}{\mathrm{Leg}}

\catcode`,\active

\catcode`\,12

\numberwithin{equation}{section}

\begin{document}
\title[Sato-Tate Distribution of $p$-adic hypergeometric functions]
{Sato-Tate Distribution of $p$-adic hypergeometric functions}

\author{Sudhir Pujahari}
\address{School of Mathematical Sciences, National Institute of Science Education and Research, Bhubaneswar, An OCC of Homi Bhabha National Institute,  P. O. Jatni,  Khurda 752050, Odisha, India.}
\email{spujahari@niser.ac.in}
\author{Neelam Saikia}
\address{Faculty of Mathematics, University of Vienna, Austria 1010.}
 \email{nlmsaikia1@gmail.com}


\begin{abstract}  
Recently Ono, Saad and the second author \cite{KHN} initiated a study of value distribution of certain families of Gaussian hypergeometric functions over large finite fields. They investigated two families of Gaussian hypergeometric functions and showed that they satisfy semicircular and Batman distributions. Motivated by their results we aim to study distributions of certain families of hypergeometric functions in the $p$-adic setting over large finite fields. In particular, we consider two and six parameters families of hypergeometric functions in the $p$-adic setting and obtain that their limiting distributions are semicircular over large finite fields. 
In the process of doing this we also express the traces of $p$th Hecke operators acting on the spaces of cusp forms of even weight $k\geq4$ and levels 4 and 8 in terms of $p$-adic hypergeometric function which is of independent interest. These results can be viewed as $p$-adic analogous of some trace formulas of \cite{ah, ah-ono, fop}.
\end{abstract}

\maketitle
\section{Introduction} 

Let $p$ be an odd prime and let $\F_q$ denote the finite field with $q=p^r$ elements. In \cite{Greene}, Greene introduced Gaussian hypergeometric functions over finite fields using Jacobi sums. One of the important fact about these functions is that they satisfy many analogous hypergeometric type identities. Recently Ono, Saad and the second author \cite{KHN} initiated a study of value distributions of certain families of these functions over random large finite fields $\F_q.$ More precisely, they investigated the distributions of the normalized values of the following two Gaussian hypergeometric functions over $\F_q$ that are defined by
$$
_2F_1(\lambda)_q:={_2F_1}\left(\begin{matrix}
\phi, & \phi\\
~ &\varepsilon
\end{matrix}\mid \lambda\right)_q:
=\frac{q}{q-1}\sum\limits_{\chi}{\phi\chi\choose\chi}
{\phi\chi\choose\chi}\chi(\lambda)
$$
and $$_3F_2(\lambda)_q:={_2F_1}\left(\begin{matrix}
\phi, & \phi, & \phi\\
~ &\varepsilon, & \varepsilon
\end{matrix}\mid \lambda\right)_q:
=\frac{q}{q-1}\sum\limits_{\chi}{\phi\chi\choose\chi}{\phi\chi\choose\chi}{\phi\chi\choose\chi}\chi(\lambda),
$$ 
where $\phi$ and $\varepsilon$ are quadratic and trivial characters of $\F_q^{\times}$ and the sums on the extreme rights run over all multiplicative characters \footnote{$\chi(0):=0$ for all multiplicative characters of $\F_q.$} of $\F_q.$ Moreover, the symbol ${A\choose B}:=\frac{B(-1)}{q}\sum\limits_{x\in\F_q}A(x)\overline{B}(1-x)$ is the normalized Jacobi sum. Ono et al. \cite{KHN} showed that for the $_2F_1(\lambda)_q$ function the limiting distribution is semicircular whereas for the $_3F_2(\lambda)_q$ function the distribution is Batman. In this paper our goal is to study similar questions for $p$-adic hypergeometric functions over random large finite fields.

McCarthy \cite{mccarthy1, mccarthy2} defined $p$-adic hypergeometric functions using $p$-adic gamma functions extending Greene's hypergeometric functions \cite{Greene} for wider classes of primes. These functions appear in the study of Frobenius trace of elliptic curves \cite{mccarthy2}, Fourier coefficients of Hecke eigen forms \cite{PS}, and proofs of supercongruence type identities \cite{FM}.

Let $\Gamma_p(\cdot)$ be the Morita's $p$-adic gamma function. Let $\omega$ be the Teichm\"{u}ller character \footnote{$\omega(t)\equiv t\pmod{p}$ for all $t\in\F_p$ and $\omega(0):=0.$} of $\mathbb{F}_p$ and $\overline{\omega}$ denote its character inverse. For $x\in\mathbb{Q}$ let $\lfloor x\rfloor$ denote the greatest integer less than or equal to $x$ and $\langle x\rangle$ denote the fractional part of $x$, satisfying $0\leq\langle x\rangle<1$. Using these notation $p$-adic hypergeometric function is defined as follows.
\begin{definition}\cite[Definition 5.1]{mccarthy2} \label{defin1}
Let $p$ be an odd prime and $t \in \mathbb{F}_p$.
For positive integer $n$ and $1\leq k\leq n$, let $a_k$, $b_k$ $\in \mathbb{Q}\cap \mathbb{Z}_p$.
Then 
\begin{align}
&_n{G}_n\left[\begin{array}{cccc}
             a_1, & a_2, & \ldots, & a_n \\
             b_1, & b_2, & \ldots, & b_n
           \end{array}\mid t
 \right]_p:=\frac{-1}{p-1}\sum_{j=0}^{p-2}(-1)^{jn}~~\overline{\omega}^j(t)\notag\\
&\times \prod\limits_{k=1}^n(-p)^{-\lfloor \langle a_k \rangle-\frac{j}{p-1} \rfloor -\lfloor\langle -b_k \rangle +\frac{j}{p-1}\rfloor}
 \frac{\Gamma_p(\langle a_k-\frac{j}{p-1}\rangle)}{\Gamma_p(\langle a_k \rangle)}
 \frac{\Gamma_p(\langle -b_k+\frac{j}{p-1} \rangle)}{\Gamma_p(\langle -b_k \rangle)}.\notag
\end{align}
\end{definition}
In this paper we aim to investigate the value distributions of certain families of these functions over large finite fields $\F_p$. Namely, we study the value distributions of the following families of $_nG_n$-functions for $n=2,6.$  
If $\lambda\in\F_p^{\times}$ and $\psi_6=\omega^{\frac{p-1}{6}}$ is a character of order 6, then for $\lambda\neq-1$ we define
$$
_2G_2(\lambda)_p:=p \psi_6(2)\phi(1+\lambda)\cdot{_2G_2}\left[\begin{array}{cc}
               \frac{2}{3}, & \frac{2}{3} \vspace{.1 cm}\\
               \frac{5}{12}, & \frac{11}{12}
             \end{array}\mid \frac{4\lambda}{(1+\lambda)^2}
            \right]_p.
$$


and
$$
_6G_6(\lambda)_p:=\phi(1+\lambda)\cdot{_6G_6}\left[\begin{array}{cccccc}
               \frac{1}{3}, & \frac{1}{3}, & \frac{2}{3}, & \frac{2}{3}, & 0, & 0 \vspace{.1 cm}\\
               \frac{1}{12}, & \frac{1}{4}, & \frac{5}{12}, & \frac{7}{12}, & \frac{3}{4}, & \frac{11}{12}
             \end{array}\mid \frac{2^6\lambda^3}{(1+\lambda)^6}
\right]_p.$$

As our first theorem we obtain the moments of values of ${_2G_2}(\lambda)_p.$
\begin{theorem}\label{moment-1}
Let $m$ be a fixed positive integer and $p\equiv1\pmod{3}$ be a prime. Then as $p\rightarrow\infty$
$$
\sum\limits_{\lambda\in\F_p}{_2G_2}(\lambda)_p^{m}=
\begin{cases}
o_{m}(p^{\frac{m}{2}+1}) & \text{if}\ \ m \  is\  \text{odd}\vspace{.2 cm}\\
\frac{(2n)!}{n!(n+1)!}p^{n+1}+o_m(p^{n+1}) & \text{if}\ \ m=2n \  is\  \text{even}.
\end{cases}
$$
\end{theorem}
We use these moments to conclude the limiting behaviour of ${_2G_2(\lambda)_p}$ as $p\rightarrow\infty.$ 
If we view the normalized values $p^{-1/2}\cdot{_2G_2(\lambda)_p}\in[-2,2]$ as random variables over $\F_p$ then we obtain the limiting distribution of this function. More precisely we have the following distribution.
\begin{corollary}\label{distribution-2G2}
If $-2\leq a<b\leq2$ then
$$
\lim\limits_{p\rightarrow\infty}\frac{|\{\lambda:\ \  p^{-1/2} {_2G_2}(\lambda)_p\in[a,b]\}|}{p}=\frac{1}{2\pi}\int_a^b\sqrt{4-t^2}~dt.
$$

\end{corollary}
We also consider these problems for the ${_6G_6(\lambda)_p}$ functions.
\begin{theorem}\label{moment-2}
Let $m$ be a fixed positive integer and $p\equiv2\pmod{3}$ be a prime. Then as $p\rightarrow\infty$
$$
\sum\limits_{\lambda\in\F_p}{_6G_6}(\lambda)_p^{m}=
\begin{cases}
o_{m}(p^{\frac{m}{2}+1}) & \text{if}\ \ m \  is\  \text{odd}\vspace{.2 cm}\\
\frac{(2n)!}{n!(n+1)!}p^{n+1}+o_m(p^{n+1}) & \text{if}\ \ m=2n \  is\  \text{even}.
\end{cases}
$$
\end{theorem}
Similarly as in Corollary \ref{distribution-2G2}, we conclude the limiting distribution of $p^{-1/2}\cdot{_6G_6(\lambda)_p}\in[-2,2]$ in the following Corollary.
\begin{corollary}\label{distribution-6G6}

If $-2\leq a<b\leq2$ then
$$
\lim\limits_{p\rightarrow\infty}\frac{|\{\lambda:\ \  p^{-1/2}\cdot{_6G_6}(\lambda)_p\in[a,b]\}|}{p}=\frac{1}{2\pi}\int_a^b
\sqrt{4-t^2}~dt.
$$

\end{corollary}

\begin{remark}
It is important to note that these results can be extended to $_2G_2(\lambda)_q$ and $_6G_6(\lambda)_q$ over finite fields $\F_q$ where $q=p^r$ using similar arguments. For simplicity we choose $q=p.$
\end{remark}

\section{Traces of Hecke operators and hypergeometric functions}
It turns out that the hypergeometric functions $_2G_2(\lambda)_p$ and $_6G_6(\lambda)_p$ can be related to the traces of Hecke operators acting on the spaces of cusp forms. 
More precisely, for positive integers $N$ and $k$, let $S(N,k)$ be the space of cusp forms of weight $k$ with respect to the congruence subgroup $\Gamma_0(N)$ and let $\Tr(\Gamma_0(N),p)$ denote the trace of the $p$th Hecke operator acting on the space 
$S (N,k)$. In a series of papers Ahlgren \cite{ah}, Ahlgren-Ono \cite{ah-ono} and 
Frechette-Ono-Papanikolas \cite{fop} studied the Eichler-Selberg type trace formulas for the $p$th Hecke operators acting on the spaces $S(N,k)$ for $N=2,4,8$ and established trace formulas using Gaussian hypergeometric functions over finite fields.
Here we obtain the following results expressing the trace formulas $\Tr(\Gamma_0(4),p)$ and $\Tr(\Gamma_0(8),p)$ using $p$-adic hypergeometric functions  $_2G_2(\lambda)_p$ and $_6G_6(\lambda)_p$. Note that both these functions vanish at $\lambda=-1,$ therefore, we define refinements of these functions. To this end we define the following two functions.
$$
{_2\widetilde{G}_2}(\lambda)_p:= {_2G_2}(\lambda)_p-\delta(1+\lambda)\cdot\Delta(p)\cdot
A_4(-1)\re\left(\frac{g(\phi A_4)g(\overline{A}_4)}{g(\phi)}\right)
$$
and
$$
{_6\widetilde{G}_6}(\lambda)_p:= {_6G_6}(\lambda)_p-\delta(1+\lambda)\cdot
\Delta(p)\cdot A_4(-1)\re\left(\frac{g(\phi A_4)g(\overline{A}_4)}{g(\phi)}\right),
$$
where $A_4$ is a character of order 4,

$\delta(x):=\begin{cases}
1, & \text{if} \ x=0\\
0, & \text{if} \ x\neq0,
\end{cases}
$
and 
$\Delta(p):=\begin{cases}
1, & \text{if} \ p\equiv1\pmod{4}\\
0, & \text{if} \ p\equiv3\pmod{4}.
\end{cases}
$

\begin{theorem}\label{Trace-case-1}
If $p\equiv1\pmod3$ is an odd prime and $k\geq4$ is even then 
$$
\Tr(\Gamma_0(4),p)=-3-\sum_{\lambda=2}^{p-1}P_k(_2\widetilde{G}_2(\lambda)_p,p)
$$
and 
$$
\Tr(\Gamma_0(8),p)=-4-\sum_{\lambda=2}^{p-1}
P_k(_2\widetilde{G}_2(\lambda^2)_p,p).
$$
\end{theorem}

\begin{theorem}\label{Trace-case-2}
If $p\equiv2\pmod3$ is an odd prime and $k\geq4$ is even then 
$$
\Tr(\Gamma_0(4),p)=-3-\sum_{\lambda=2}^{p-1}P_k(_6\widetilde{G}_6
(\lambda)_p,p)
$$
and 
$$
\Tr(\Gamma_0(8),p)=-4-\sum_{x=2}^{p-1}P_k(_6\widetilde{G}_6(\lambda^2)_p,p).
$$
\end{theorem}

\section{Gauss sums and $p$-adic gamma function}
In this section we discuss some important theorems namely Gross-Koblitz formula and Davenport-Hasse relation. We also recall some basic results in Gauss sums and $p$-adic gamma functions. We begin by recalling the orthogonality relations satisfied by multiplicative characters. Let $\widehat{\mathbb{F}_p^\times}$ denote the cyclic group of multiplicative characters of $\F_p^{\times}.$

\begin{lemma}\emph{(\cite[Chapter 8]{ireland}).}\label{lemma-3} If $\chi$ is a multiplicative character \footnote{For all multiplicative characters $\chi$ including the trivial character $\chi(0):=0.$} of $\mathbb{F}_p^\times,$ then

\smallskip
\noindent                               
(1) $\sum\limits_{\chi\in \widehat{\mathbb{F}_p^\times}}\chi(x)~~=\left\{
                            \begin{array}{ll}
                              p-1 & \hbox{if~~ $x=1$;} \\
                              0 & \hbox{if ~~$x\neq1$.}
                            \end{array}
                          \right.$

\smallskip
\noindent  
(2)  $\sum\limits_{x\in \mathbb{F}_p}\chi(x)~~=\left\{
                            \begin{array}{ll}
                              p-1 & \hbox{if~~ $\chi=\varepsilon$;} \\
                              0 & \hbox{if ~~$\chi\neq\varepsilon$.}
                            \end{array}
                          \right.$

\end{lemma}

Let $\zeta_p$ be a primitive $p$th roots of unity $\chi$ be a multiplicative character ${\mathbb{F}_p^\times}.$ Then the Gauss sum is defined by
$$
g(\chi):=\sum\limits_{x\in \mathbb{F}_p}\chi(x)\zeta_p^{x}.
$$

\begin{theorem}\emph{(\cite[Davenport-Hasse Relation]{evans}).}\label{DH}
Let $n$ be a positive integer and let $p$ be a prime such that $p\equiv 1 \pmod{n}$. For multiplicative characters
$\chi, \psi \in \widehat{\mathbb{F}_p^\times}$, we have
\begin{align}
\prod\limits_{\chi^n=\varepsilon}g(\chi \psi)=-g(\psi^n)\psi(n^{-n})\prod\limits_{\chi^n=\varepsilon}g(\chi).\notag
\end{align}
\end{theorem}
Let $\mathbb{Z}_p$ and $\mathbb{Q}_p$ denote the ring of $p$-adic integers and the field of $p$-adic numbers, respectively.
Let $\overline{\mathbb{Q}_p}$ be the algebraic closure of $\mathbb{Q}_p$ and $\mathbb{C}_p$ be the completion of $\overline{\mathbb{Q}_p}$.
For a positive integer $n$
the $p$-adic gamma function $\Gamma_p(n)$ is defined as
$$
\Gamma_p(n):=(-1)^n\prod\limits_{0<j<n,p\nmid j}j.\notag
$$
The domain of definition can be extended to all $x\in\mathbb{Z}_p$ by simply setting $\Gamma_p(0):=1$ and for $x\neq0$
$$
\Gamma_p(x):=\lim_{x_n\rightarrow x}\Gamma_p(x_n),\notag
$$
where $x_n$ is any sequence of positive integers $p$-adically approaching $x.$
The product formula given in \eqref{prod} can be understood as a $p$-adic analogue of Davenport-Hasse relation. 
If $n\in\mathbb{Z}^{+}$, $p\nmid n$ and $x=\frac{r}{p-1}$ with $0\leq r\leq p-1$, then
\begin{align}\label{prod}
\prod_{h=0}^{n-1}\Gamma_p\left(\frac{x+h}{n}\right)=\omega(n^{(1-x)(1-p)})
\Gamma_p(x)\prod_{h=1}^{n-1}\Gamma_p\left(\frac{h}{n}\right).
\end{align}
To this end we recall Gross-Koblitz formula which relates Gauss sums to $p$-adic gamma functions. Let $\pi \in \mathbb{C}_p$ be the fixed root of the polynomial $x^{p-1} + p$, which satisfies the congruence condition $\pi \equiv \zeta_p-1 \pmod{(\zeta_p-1)^2}$. 
\begin{theorem}\emph{\cite[Gross-Koblitz]{gross}.}\label{GK} For $a\in \mathbb{Z}$,
\begin{align}
g(\overline{\omega}^a)=-\pi^{(p-1)\langle\frac{a}{p-1} \rangle}\Gamma_p\left(\left\langle \frac{a}{p-1} \right\rangle\right).\notag
\end{align}
\end{theorem}

\begin{lemma}\emph{\cite[Lemma 4.1]{mccarthy2}.}\label{lemma-5}
Let $p$ be a prime and $0\leq j\leq p-2$. For $t\geq 1$ with $p\nmid t$, we have
\begin{align}
\omega(t^{tj})\Gamma_p\left(\left\langle \frac{tj}{p-1}\right\rangle\right)
\prod\limits_{h=1}^{t-1}\Gamma_p\left(\left\langle\frac{h}{t}\right\rangle\right)
=\prod\limits_{h=0}^{t-1}\Gamma_p\left(\left\langle\frac{h}{t}+\frac{j}{p-1}\right\rangle\right),\notag
\end{align}
and
\begin{align}
\omega(t^{-tj})\Gamma_p\left(\left\langle\frac{-tj}{p-1}\right\rangle\right)
\prod\limits_{h=1}^{t-1}\Gamma_p\left(\left\langle \frac{h}{t}\right\rangle\right)
=\prod\limits_{h=0}^{t-1}\Gamma_p\left(\left\langle\frac{h}{t}-\frac{j}{p-1}\right\rangle \right).\notag
\end{align}
\end{lemma}

\section{proofs of theorems}
For $\lambda\neq0,1$ let $E_{\lambda}^{\Leg}:\ \ y^2=x(x-1)(x-\lambda)$ be the Legendre normal form of elliptic curve over $\F_p$ and  
$$a_p(\lambda):=p+1-|E_{\lambda}^{\Leg}(\F_p)|=-\sum\limits_{x\in\F_p}{\phi(x(x-1)(x-\lambda))}$$ be the trace of Frobenius of the elliptic curve 
$E_{\lambda}^{\Leg}.$ Our main idea is to express the functions $_2G_2(\lambda)_p$ and $_6G_6(\lambda)_p$ in terms of the traces of Frobenius $a_p(\lambda).$ Therefore,
to obtain the asymptotic formulas of power moments of these functions it is sufficient if we have asymptotic formulas for the power moments of $a_p(\lambda).$ In \cite{KHN} Ono et al. computed the asymptotic formulas for power moments of the functions $_2F_1(\lambda)_p.$ Here we reformulate their formulas in terms of power moments of $a_p(\lambda)$ in the following theorem. 

\begin{theorem}\label{2F1moment}
Let $m$ be a fixed positive integer and $p>3$ be a prime. Then as $p\rightarrow\infty$
\begin{equation}
    \sum\limits_{\lambda\neq0,1}a_p(\lambda)^m=
\begin{cases}
o_{m}(p^{\frac{m}{2}+1}) & \text{if}\ \ m \  is\  \text{odd}\vspace{.2 cm}\\
\frac{(2n)!}{n!(n+1)!}p^{n+1}+o_m(p^{n+1}) & \text{if}\ \ m=2n \  is\  \text{even}.
\end{cases}
\end{equation}
\end{theorem}
\begin{proof}
The proof follows easily by making use of Theorem 1 of \cite{ono}, Theorem 1.1 of \cite{KHN} and the value $_2F_1(1)_p=-\phi(-1)/p.$
\end{proof}
Theorem \ref{2F1moment} is a refinement of a classical theorem of Birch \cite[Theorem 1]{Birch} (when restricted to Legendre normal elliptic curves) that established even moments of all elliptic curves over finite fields. Recently Bringmann, Kane, and the first author \cite{KP-1, KP-2} have refined Birch's result to arithmetic progressions. 

\begin{proof}[Proof of Theorem \ref{moment-1}]
By the definition we can write
\begin{align}
{_2G_2}(\lambda)_p&=\frac{p\cdot\psi_6(2)\phi(1+\lambda)}{(1-p)\Gamma_p(\frac{1}{12})\Gamma_p(\frac{7}{12})
\Gamma_p(\frac{2}{3})^2}
     \sum_{j=0}^{p-2}\overline{\omega}^j\left(\frac{4\lambda}{(1+\lambda)^2}\right)\notag\\
&\times (-p)^{-\left\lfloor\frac{1}{12}+\frac{j}{p-1}
\right\rfloor
-\left\lfloor\frac{7}{12}+\frac{j}{p-1}\right\rfloor
-2\left\lfloor\frac{2}{3}-\frac{j}{p-1}\right\rfloor}~
\Gamma_p\left(\left\langle\frac{1}{12}+\frac{j}{p-1}
\right\rangle\right)
\Gamma_p\left(\left\langle\frac{7}{12}+\frac{j}{p-1}
\right\rangle\right)\notag\\
&\times\Gamma_p\left(\left\langle\frac{2}{3}-\frac{j}{p-1}
\right\rangle\right)^2.\notag
\end{align}

Using \eqref{prod} we obtain 
$
\Gamma_p(\frac{1}{12})\Gamma_p(\frac{7}{12})
\Gamma_p(\frac{2}{3})^2=\phi(-2).
$

Then substituting this value and taking the transformation $j\rightarrow j+\frac{p-1}{3}$ in the above sum we have

\begin{align}\label{eqn-1}
{_2G_2}(\lambda)_p&=\frac{p\cdot\phi(-2)\psi_6(2) \phi(1+\lambda)}{1-p}\overline{\psi}_3
\left(\frac{\lambda}{(1+\lambda)^2}\right)
     \sum_{j=0}^{p-2}\overline{\omega}^j\left(\frac{4\lambda}{(1+\lambda)^2}\right)\notag\\
&\times (-p)^{-\left\lfloor\frac{5}{12}+\frac{j}{p-1}
\right\rfloor
-\left\lfloor\frac{11}{12}+\frac{j}{p-1}\right\rfloor
-2\left\lfloor\frac{1}{3}-\frac{j}{p-1}\right\rfloor}~
\Gamma_p\left(\left\langle\frac{5}{12}+\frac{j}{p-1}
\right\rangle\right)
\Gamma_p\left(\left\langle\frac{11}{12}+\frac{j}{p-1}
\right\rangle\right)\notag\\
&\times\Gamma_p\left(\left\langle\frac{1}{3}-\frac{j}{p-1}
\right\rangle\right)^2,
\end{align}

where $\psi_3=\omega^{\frac{p-1}{3}}$ is a character of order 3.
If we put $n=2$ and 

$x=\langle\frac{5}{6}+\frac{2j}{p-1}\rangle$ in \eqref{prod}, then 
we have

$$\Gamma_p\left(\left\langle\frac{5}{6}+\frac{2j}{p-1}\right\rangle\right)=
\overline{\omega}^{j}(4)\psi_6(2)
\frac{\Gamma_p(\langle\frac{5}{12}+\frac{j}{p-1}\rangle)\Gamma_p(\langle\frac{11}{12}+\frac{j}{p-1}\rangle)}{\Gamma_p(\frac{1}{2})}.$$ 
 
 Also it is easy to verify that 

$$\left\lfloor\frac{5}{6}+\frac{2j}{p-1}\right\rfloor
=\left\lfloor\frac{5}{12}+\frac{j}{p-1}\right\rfloor
+\left\lfloor\frac{11}{12}+\frac{j}{p-1}\right\rfloor.$$

Using these two facts in \eqref{eqn-1} we obtain

\begin{align}\label{eqn-2}
{_2G_2}(\lambda)_p&=\frac{p\Gamma_p(\frac{1}{2})\phi(1+\lambda)}{1-p}\phi(-2)\overline{\psi}_3
\left(\frac{\lambda}{(1+\lambda)^2}\right)
     \sum_{j=0}^{p-2}\overline{\omega}^j\left(\frac{\lambda}{(1+\lambda)^2}\right)\notag\\
&\times (-p)^{-\left\lfloor\frac{5}{6}+\frac{2j}{p-1}\right\rfloor
-2\left\lfloor\frac{1}{3}-\frac{j}{p-1}\right\rfloor}~
\Gamma_p\left(\left\langle\frac{5}{6}+\frac{2j}{p-1}
\right\rangle\right)
\Gamma_p\left(\left\langle\frac{1}{3}-\frac{j}{p-1}
\right\rangle\right)^2.
\end{align}

Now, Gross-Koblitz formula allow us to write

$$
{_2G_2}(\lambda)_p=\frac{\phi(2)\phi(1+\lambda)}{(1-p)}\overline{\psi}_3
\left(\frac{\lambda}{(1+\lambda)^{2}}\right)
\sum_{j=0}^{p-2}\overline{\omega}^j\left(\frac{\lambda}
{(1+\lambda)^2}\right)
\frac{g(\psi_6^5\overline{\omega}^{2j})g(\psi_3\omega^j)^2}
{g(\phi)}.
$$

Replacing $\overline{\omega}^j$ by $\overline{\omega}^j\psi_3$ we obtain

\begin{align}\label{compare}
{_2G_2}(\lambda)_p&=\frac{\phi(2)\phi(1+\lambda)}{(1-p)}
\sum_{j=0}^{p-2}\overline{\omega}^j\left(\frac{\lambda}
{(1+\lambda)^2}\right)
\frac{g(\phi\overline{\omega}^{2j})g(\omega^j)^2}
{g(\phi)}\\
&=\frac{\phi(2)\phi(1+\lambda)}{(1-p)g(\phi)}\sum_{u,v,y\in\F_p}\phi(u)\zeta_p^{u+v+y}\sum_{j=0}^{p-2}\omega^j\left(\frac{yv(1+\lambda)^2}{u^2\lambda}\right).\notag
\notag\end{align}

Using the orthogonality of multiplicative characters and then replacing $u$ by $-(1+\lambda)uy$ we deduce that

$$
{_2G_2}(\lambda)_p=-\phi(-2)\sum\limits_{u\in\F_p}\phi(u(1-u)(1-\lambda u)).
$$
Again transforming $u$ by $1/u$ we obtain 
\begin{align}\label{relation-1}
{_2G_2}(\lambda)_p=-\phi(-2)\sum\limits_{u\in\F_p}\phi(u(u-1)
(u-\lambda)).
\end{align}

Therefore, for $\lambda\neq0,\pm1$ we have
\begin{equation}\label{hypergeometric-trace-1}
    _2G_2(\lambda)_p=\phi(-2)a_p(\lambda).
\end{equation}
Using this relation we write

\begin{align}\label{final-eq-1}
    \sum\limits_{\lambda\in\F_p} {_2G_2}(\lambda)_p^m= {_2G_2}(1)_p^m
-\phi(-2)^m a_p(-1)^m
+\phi(-2)^m\sum\limits_{\lambda\neq0,1} a_p(\lambda)^m.
\end{align}

By \eqref{relation-1} and second part of Lemma \ref{lemma-3} it is easy to see that
\begin{align}\label{final-eq-2}
{_2G_2}(1)_p=\phi(-2).
    \end{align}
    
    Using \cite[eqn (2.10)]{Greene} and the orthogonality relation Lemma \ref{lemma-3} one can verify that 
    \begin{align}\label{AP}
     a_p(-1)^m=\begin{cases}
0, & \text{if}\ p\equiv3\pmod{4}\\
(-1)^m A_4^m(-1)\left(\frac{g(\phi A_4)g(\overline{A}_4)}{g(\phi)}
+\frac{g(\phi \overline{A}_4)g(A_4)}{g(\phi)}\right)^m, &
\text{if}\ p\equiv1\pmod{4},
\end{cases}   
    \end{align}
where $A_4$ is a character of order 4. For any multiplicative character $\chi$ we have $|g(\chi)|=\sqrt{p}.$ Using this in the above identity we obtain

\begin{align}\label{final-eq-3}
 a_p(-1)^m=o(p^{m/2+1}).   
\end{align}

Finally, applying Theorem \ref{2F1moment}, \eqref{final-eq-2} and \eqref{final-eq-3} in \eqref{final-eq-1} we conclude the result.

\end{proof}

\begin{proof}[Proof of Theorem \ref{moment-2}]
Let $H(\lambda)={_6G_6}\left[\begin{array}{cccccc}
\frac{1}{12}, & \frac{1}{4}, & \frac{5}{12}, & \frac{7}{12}, & \frac{3}{4}, & \frac{11}{12}\vspace{.1 cm}\\
\frac{1}{3}, & \frac{1}{3}, & \frac{2}{3}, & \frac{2}{3}, & 0, & 0
             \end{array}\mid \frac{(1+\lambda)^6}{2^6\lambda^3}
\right]_p.$
Using definition we write
\begin{align}
H(\lambda)&=\frac{1}{1-p}\sum_{j=0}^{p-2}{\omega}^{j}\left(\frac{2^6\lambda^3}{(1+\lambda)^6}\right)
\pi^{(p-1)S_j}\notag\\
&\times\prod_{h=1;~h~odd}^{11}\frac{\Gamma_p(\langle\frac{h}{12}-
\frac{j}{p-1}\rangle)}{\Gamma_p(\langle\frac{h}{12}\rangle)}\times
\frac{\Gamma_p(\langle\frac{1}{3}+\frac{j}{p-1}\rangle)^2~\Gamma_p(\langle\frac{2}{3}+\frac{j}{p-1}\rangle)^2~\Gamma_p(\langle\frac{j}{p-1}\rangle)^2}
{\Gamma_p(\langle\frac{1}{3}\rangle)^2~\Gamma_p(\langle\frac{2}{3}\rangle)^2},
\end{align}

where 
$S_j=-2\left\lfloor\frac{j}{p-1}\right\rfloor-2\left\lfloor\frac{1}{3}+\frac{j}{p-1}\right\rfloor
-2\left\lfloor\frac{2}{3}+\frac{j}{p-1}\right\rfloor
-\sum\limits_{h=1;~h~odd}^{12}\left\lfloor\frac{h}{12}-\frac{j}{p-1}\right\rfloor.$

Now applying Lemma \ref{lemma-5} we deduce that 

$$
H(\lambda)=\frac{1}{1-p}\sum\limits_{j=0}^{p-2}{\omega}^{j}
\left(\frac{\lambda^3}{2^{12}(1+\lambda)^6}\right)
\pi^{(p-1)S_j}
\frac{\Gamma_p(\langle\frac{-12j}{p-1}\rangle)\Gamma_p(\langle\frac{3j}{p-1}\rangle)\Gamma_p(\langle\frac{3j}{p-1}\rangle)}
{\Gamma_p(\langle\frac{-6j}{p-1}\rangle)},
$$

For $1\leq j\leq p-2$ it is easy to very that 
$
\left\lfloor\frac{-12j}{p-1}\right\rfloor
=1+\sum\limits_{h=0}^{11}\left\lfloor\frac{h}{12}-\frac{j}{p-1}\right\rfloor
$ and 

$
\left\lfloor\frac{-6j}{p-1}\right\rfloor
=1+\sum\limits_{h=0}^{5}\left\lfloor\frac{h}{6}-\frac{j}{p-1}\right\rfloor.
$
Moreover, 
$
\left\lfloor\frac{-3j}{p-1}\right\rfloor=
\left\lfloor\frac{j}{p-1}\right\rfloor+\left\lfloor\frac{1}{3}+\frac{j}{p-1}\right\rfloor+\left\lfloor\frac{2}{3}+\frac{j}{p-1}\right\rfloor.
$

If we use of these three relations in the exponent of $\pi$ present in $G(\lambda)$ then the Gross-Koblitz formula allows to write

$$
H(\lambda)=\frac{1}{(1-p)}\sum_{j=0}^{p-2}
\frac{g({\omega}^{12j})g(\overline{\omega}^{3j})^2}{g({\omega}^{6j})}
{\omega}^{j}\left(\frac{\lambda^3}{2^{12}(1+\lambda)^6}\right).
$$

By making use of Davenport-Hasse relation one can have
\begin{equation}\label{D-H}
    g({\omega}^{6j})g(\phi{\omega}^{6j})
=g({\omega}^{12j}){\omega}^{6j}(2^{-2})g(\phi).
\end{equation}
Therefore, using \eqref{D-H} we have

$$
H(\lambda)=\frac{1}{(1-p)}\sum_{j=0}^{p-2}
\frac{g(\phi{\omega}^{6j})g(\overline{\omega}^{3j})^2}{g(\phi)}
{\omega}^{j}\left(\frac{\lambda^3}{(1+\lambda)^6}\right).
$$
Since $p\equiv2\pmod{3}$ therefore transforming $j\rightarrow -j/3$ we obtain

$$
H(\lambda)=\frac{1}{(1-p)}\sum_{j=0}^{p-2}
\frac{g(\phi\overline{\omega}^{2j})g(\omega^{j})^2}{g(\phi)}
\overline{\omega}^{j}\left(\frac{\lambda}{(1+\lambda)^2}\right).
$$
Now observe that the above summation is equal to the summation involved in \eqref{compare} up to a scalar multiple.  Therefore following similar steps 
we obtain

$$
\phi(1+\lambda)H(\lambda)=\phi(-1)\cdot a_p(\lambda).
$$
Therefore, if $\lambda\neq0,\pm1$ then
\begin{equation}\label{hypergeometric-trace-2}
    _6G_6(\lambda)_p=\phi(-1)\cdot a_p(\lambda).
\end{equation}

The rest of the proof relies on similar arguments as in the proof of Theorem \ref{moment-1}.

\end{proof}

\begin{proof}[Proof of Theorem \ref{Trace-case-1}]
If $\lambda\neq0,\pm1,$ then \eqref{hypergeometric-trace-1} gives
$_2\widetilde{G}_2(\lambda)_p=a_p(\lambda).$ Also by \eqref{AP} we have 
$_2\widetilde{G}_2(-1)_p=a_p(-1).$ Therefore, for $\lambda\neq0,1$

$$
_2\widetilde{G}_2(\lambda)_p=a_p(\lambda).
$$

If $k\geq4$ is even then Proposition 2.1 of \cite{fop} gives
\begin{align}\label{Hecke-Trace-1}
\Tr(\Gamma_0(4), p)&=-3-\sum_{\lambda=2}^{p-1}P_k(a_p(\lambda),p),\\
\label{Hecke-Trace-2}
\Tr(\Gamma_0(8), p)&=-4-\sum_{\lambda=2}^{p-1}P_k(a_p(\lambda^2),p).
\end{align}
Therefore, we conclude the results by replacing $a_p(\lambda)$ by $_2\widetilde{G}_2(\lambda)_p$ in the above two trace formulas.
\end{proof}

\begin{proof}[Proof of Theorem \ref{Trace-case-2}]
The proof follows similarly as in the proof of Theorem \ref{Trace-case-1}. The only requirement is to show for $\lambda\neq0,1$
$$
_6\widetilde{G}_6(\lambda)_p=a_p(\lambda).
$$
This is straightforward by revisiting \eqref{hypergeometric-trace-2} and the definition of $_6\widetilde{G}_6(-1)_p.$ 

\end{proof}
\section{Acknowledgements}
The second author thanks the Austrian Science Fund FWF grant P32305 for the support.

\end{document}